\documentclass[12pt,a4paper,dvips]{amsart}
\usepackage[text={16cm,24.5cm}]{geometry}

\usepackage{amsfonts, amsmath, amssymb, amsthm}
\usepackage[mathscr]{eucal} 
\usepackage{hyperref}

\usepackage{graphicx,color}
\usepackage{booktabs}
\usepackage{bbm}
\usepackage[latin1]{inputenc}
\usepackage[T1]{fontenc}
\usepackage{overpic}

\newcommand{\RR}{{\mathbb{R}}}

\newcommand{\Sph}{{\mathbb{S}}}

\newcommand{\SetOf}[2]{\left\{#1\vphantom{#2}\,\right.\left|\,\vphantom{#1}#2\right\}}
\newcommand{\smallSetOf}[2]{\{#1\,|\,#2\}}
\newcommand{\bigSetOf}[2]{\bigl\{#1\,\big|\,#2\bigr\}}

\DeclareMathOperator{\aff}{aff}
\DeclareMathOperator{\conv}{conv}
\DeclareMathOperator{\pos}{pos}

\newcommand\cprime{$'$}

\newcommand{\transpose}[1]{{#1}^{\sf T}}

\newcommand{\SplitComplex}[1]{{\rm Split}(#1)}
\newcommand{\WeakSplitComplex}[1]{{\rm Split}^{\rm w}(#1)}
\DeclareMathOperator{\Vertices}{Vert}
\DeclareMathOperator{\interior}{int}

\newcommand{\join}{*}

\newcommand{\SecondaryFan}[1]{{\rm SecFan}(#1)}
\newcommand{\SecondaryFanOne}[1]{{\rm SecFan}'(#1)}
\newcommand{\SecondaryPolytope}[1]{{\rm SecPoly}(#1)} 
\newcommand{\ChamberComplex}[1]{{\rm Chamber}(#1)}
\newcommand{\GaleDual}[1]{{\rm Gale}(#1)}

{\end{list}}

\theoremstyle{plain}
\newtheorem{theorem}{Theorem}
\newtheorem{proposition}[theorem]{Proposition}
\newtheorem{corollary}[theorem]{Corollary}
\newtheorem{lemma}[theorem]{Lemma}

\theoremstyle{definition}

\newtheorem{example}[theorem]{Example}

\newtheorem{remark}[theorem]{Remark}

\graphicspath{ {mpost/} {eps/} }

\begin{document}

\title{Totally Splittable Polytopes}

\author[Herrmann and Joswig]{Sven Herrmann \and Michael Joswig}
\address{Sven Herrmann, Fachbereich Mathematik, TU Darmstadt, 64289 Darmstadt, Germany}
\email{sherrmann@mathematik.tu-darmstadt.de}
\address{Michael Joswig, Fachbereich Mathematik, TU Darmstadt, 64289 Darmstadt, and
  Institut für Mathematik, TU Berlin, 10623 Berlin, Germany}
\email{joswig@math.tu-berlin.de}
\thanks{Sven Herrmann is supported by a Graduate Grant of TU Darmstadt. Research by
  Michael Joswig is supported by DFG Research Unit ``Polyhedral Surfaces''.}

\date{\today}

\begin{abstract}
  A \emph{split} of a polytope is a (necessarily regular) subdivision with exactly two
  maximal cells.  A polytope is \emph{totally splittable} if each triangulation (without
  additional vertices) is a common refinement of splits.  This paper establishes a
  complete classification of the totally splittable polytopes.
\end{abstract}

\maketitle

\section{Introduction}

Splits (of hypersimplices) first occurred in the work of Bandelt and Dress on
decompositions of finite metric spaces with applications to phylogenetics in algorithmic
biology~\cite{BandeltDress92}.  This was later generalized to a result on arbitrary
polytopes by Hirai~\cite{Hirai06} and the authors~\cite{HerrmannJoswig08}.  While many
polytopes do not admit a single split, the purpose of this paper is to study polytopes
with very many splits.

The set of all regular subdivisions of a polytope $P$, partially ordered by refinement,
has the structure of the face lattice of a polytope, the \emph{secondary polytope} of $P$
introduced by Gel$'$fand, Kapranov, and Zelevinsky~\cite{GKZ94}.  The vertices of the
secondary polytope correspond to the regular triangulations, while the facets correspond
to the regular coarsest subdivisions.  There is a host of knowledge on triangulations of
polytopes \cite{Triangulations}, but information on coarsest subdivisions is scarce.
Splits are obviously coarsest subdivisions and moreover known to be regular.  So they
correspond to facets of the secondary polytope.  The total splittability of $P$ is
equivalent to the property that each facet of the secondary polytope of $P$ arises from a
split.  Via a \emph{compatibility} relation the splits of a polytope form an abstract
simplicial complex.  For instance, for the hypersimplices $\Delta(d,n)$ this turns out to
be a subcomplex of the \emph{Dressian} Dr$(d,n)$ which is an outer approximation (in terms
of matroid decompositions) of the tropical variety arising from the Grassmannian of
$d$-planes in $n$-space; see \cite[Theorem~7.8]{HerrmannJoswig08} and
\cite{HerrmannJensenJoswigSturmfels09}.

As can be expected the assumption of total splittability restricts the combinatorics of
$P$ drastically.  We prove that the totally splittable polytopes are the simplices, the
polygons, the regular crosspolytopes, the prisms over simplices, or joins of these.
Interestingly, our classification seems to coincide with those infinite families of
polytopes for which the secondary polytopes are known.  This suggests that, in order to
derive more detailed information about the secondary polytopes of other polytopes, it is
crucial to systematically investigate coarsest subdivisions other than splits.  Such a
task, however, is beyond the scope of this paper.

This is how our proof (and thus the paper) is organized: It will frequently turn out to be
convenient to phrase facts in terms of a Gale dual of a polytope.  Hence we begin our
paper with a short introduction to Gale duality and chamber complexes.  The first
important step towards the classification is the easy Proposition~\ref{prop:vertex-splits}
which shows that the neighbors of a vertex of a totally splittable polytope must span an
affine hyperplane.  Then the following observation turns out to be useful: Whenever $P$ is
a prism over a $(d-1)$-simplex or a $d$-dimensional regular crosspolytope with $d\ge 3$,
there is no place for a point $v$ outside $P$ such that $\conv(P\cup\{v\})$ is totally
splittable.  In this sense, prisms and crosspolytopes are \emph{maximally totally
  splittable}.  It is clear that the case of $d=2$ is quite different; and it is one
technical difficulty in the proof to intrinsically distinguish between polygons and higher
dimensional polytopes.  The next step is a careful analysis of the Gale dual of a totally
splittable polytope which makes it possible to recognize a potential decomposition as a
join.  A final reduction argument allows one to concentrate on maximally totally
splittable factors, which then can be identified again via their Gale duals.

We are indebted to the anonymous referees for very careful reading which lead to several
improvements in the exposition.

\section{Splits and Gale Duality}

Let $V$ be a configuration of $n\ge d+1$ (not necessarily distinct) non-zero vectors in
$\RR^{d+1}$ which linearly spans the whole space.  Often we identify $V$ with the
$n\times(d+1)$-matrix whose rows are the points in $V$, and our assumption says that the
matrix $V$ has full rank $d+1$.  Such a vector configuration gives rise to an \emph{oriented
matroid} in the following way: For a linear form $a\in(\RR^{d+1})^\star$ we have a
\emph{covector} $C^\star \in \{0,+,-\}^V$ by
\[
C^\star(v) \ := \
\begin{cases}
  0 & \text{if }av=0 \, , \\
  + & \text{if }av>0 \, , \\
  - & \text{if }av<0 \, .
\end{cases}
\]
For $\epsilon\in\{0,+,-\}$ we let $C^\star_\epsilon:=\smallSetOf{v\in
  V}{C^\star(v)=\epsilon}$, and we call the multiset $C^\star_+\cup C^\star_-$ the
\emph{support} of~$C^\star$.  Occasionally, the complement $C^\star_0$ will be called the
\emph{cosupport} of $C^\star$.  A covector whose support is minimal with respect to
inclusion of multisets is a \emph{cocircuit}; equivalently, its cosupport is maximal.
Dually, $C\subset \{0,+,-\}^V$ is called a \emph{vector} of $V$ if the linear dependence
\[
\sum_{v\in C_+}\lambda_v v \ = \ \sum_{v\in C_-}\lambda_v v
\]
holds for some coefficients $\lambda_v>0$; here $C_\epsilon$ is defined as for the
cocircuits.  The vectors with minimal support are the \emph{circuits}. Note that a point
configuration defines the circuits and cocircuits only up to a sign reversal.
Occasionally, we will speak of ``unique'' (co-)circuits with given properties, and in
these cases we always mean uniqueness up to such a reversal of the signs.  See monograph
\cite{BLSWT} for all details and proofs of properties of oriented matroids.

Now consider an $n\times(n-d-1)$-matrix $V^\star$ of full rank $n-d-1$ satisfying
$\transpose{V}V^\star=0$; that is, the columns of $V^\star$ form a basis of the kernel of
$\transpose{V}$.  Then the configuration of row vectors of $V^\star$ is called a
\emph{Gale dual} of~$V$.  Any Gale dual of $V$ is uniquely determined up to affine
equivalence.  Each vector $v\in V$ corresponds to a row vector $v^\star$ of $V^\star$,
called the \emph{vector dual} to $v$.  Throughout we will assume that all dual vectors are
either zero or have unit Euclidean length.  If $v^\star$ is zero then all vectors other
than $v$ span a linear hyperplane not containing $v$. We call $V$ \emph{proper} if
$V^\star$ does not contain any zero vectors. In the primal view, this means that $\conv V$ is
not a pyramid.  For the remainder of this section we will assume that $V$ is proper whence
$V^\star$ can be identified with a configuration of $n$ points on the unit sphere
$\Sph^{n-d-2}$.  Notice that these $n$ points are not necessarily pairwise distinct.
Repetitions may occur even if the vectors in $V$ are pairwise distinct.

The connection between Gale duality and oriented matroids is the following: The circuits
of $V$ are precisely the cocircuits of $V^\star$ and conversely.  We define the
\emph{oriented matroid} of $V$ as its set of cocircuits.  Affinely equivalent vector
configurations have the same oriented matroid, but the converse does not hold.

Now let $P$ be a $d$-dimensional polytope in $\RR^d$ with $n$ vertices.  By homogenizing
the vertices $\Vertices P$, we obtain a configuration $V_P$ of $n$ non-zero vectors in
$\RR^{d+1}$ which linearly spans the whole space.  The cocircuits of $V_P$ are given by the linear
hyperplanes spanned by vectors in $V_P$.  The vector configuration $V_P$ is proper if and
only if $P$ is not a pyramid, and we will assume that this is the case.  The \emph{Gale
  dual} of $P$ is the spherical point configuration $\GaleDual{P}:=V_P^\star$, which again
is unique up to (spherical) affine equivalence.

\begin{figure}[htb]
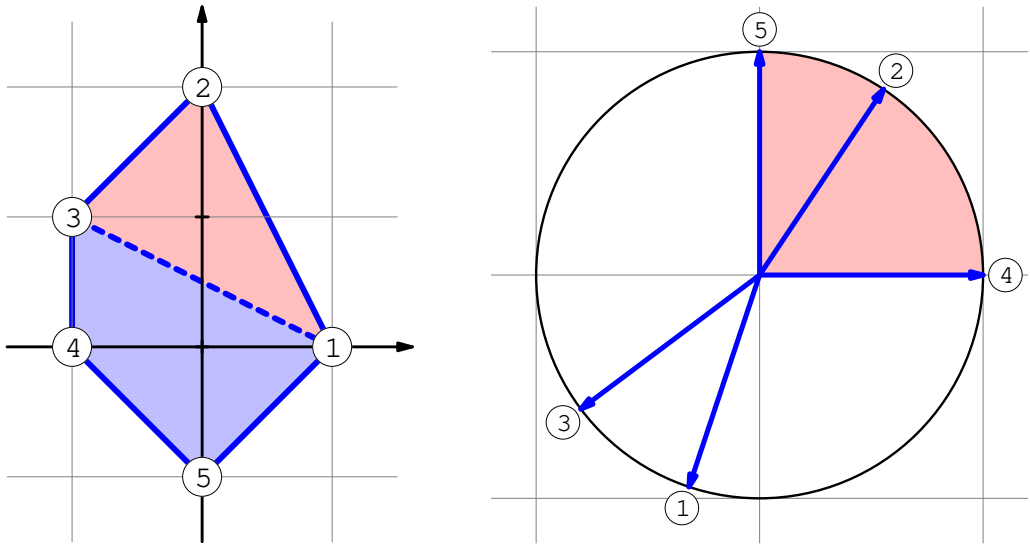

  \includegraphics[height=.45\textwidth]{pentagon.0} \qquad
  \includegraphics[height=.45\textwidth]{pentagon.1}
  \caption{Pentagon and Gale dual.  Corresponding vertices and dual vectors are labeled
    alike.}
  \label{fig:pentagon}
\end{figure}

\begin{example}\label{ex:pentagon:1}
  The matrices
  \[
  V \ := \
  \begin{pmatrix}
    1 & 1 & 0 \\
    1 & 0 & 2 \\
    1 & -1 & 1 \\
    1 & -1 & 0 \\
    1 & 0 & -1
  \end{pmatrix}
  \quad \text{and} \quad
  V^\star \ := \
  \begin{pmatrix}
    -1/3 & -1 \\
    2/3  & 1 \\
    -4/3 & -1 \\
    1 & 0 \\
    0 & 1
  \end{pmatrix}
  \]
  are Gale duals of each other.  The rows of the matrix $V$ are the homogenized vertices
  of the pentagon shown to the left in Figure~\ref{fig:pentagon}.  The Gale dual obtained
  from projecting $V^\star$ to $\Sph^1$ is shown to the right.
\end{example}

We are interested in polytopal subdivisions of our polytope $P$ and intend to study them
via Gale duality.  This requires the introduction of some notation.  A polytopal
subdivision of~$P$ is \emph{regular} if it is induced by a lifting function on the
vertices of~$P$.  The set of all lifting functions $\lambda\in\RR^n$ inducing a fixed
regular subdivision $\Sigma_\lambda$ is a relatively open polyhedral cone in $\RR^n$, the
\emph{secondary cone} of $\Sigma_\lambda$.  The set of all secondary cones forms a
polyhedral fan, the \emph{secondary fan} $\SecondaryFan{P}$.  It turns out that the
secondary fan is the normal fan of a polytope of dimension $n-d-1$, and any such polytope
is a \emph{secondary polytope} of $P$, that is the secondary polytope
$\SecondaryPolytope{P}$ is defined only up to normal equivalence.  The vertices of
$\SecondaryPolytope{P}$ correspond to the regular triangulations of $P$.  The reduction in
dimension comes from the fact that all the secondary cones in $\SecondaryFan{P}$ have a
$(d+1)$-dimensional lineality space in common.  By factoring out this lineality space and
intersecting with the unit sphere one obtains the spherical polytopal complex
$\SecondaryFanOne{P}$ in $\Sph^{n-d-2}$.  It is dual to the boundary complex of the
secondary polytope.

Now fix a Gale dual $G:=\GaleDual{P}$.  Each subset $I\subseteq[n]$ corresponds to a set
of (homogenized) vertices $V_I$.  We set $I^\star:=[n]\setminus I$ and
$V^\star_I:=\smallSetOf{v_i^\star}{i\in I}$.  Then the set $V_I$ affinely spans $\RR^d$ if and
only if the duals of the complement, that is, the set
\[
V^\star_{I^\star} \ = \ \SetOf{v_i^\star}{i\in[n]\setminus I}
\]
is linearly independent.  In particular, for each $d$-dimensional simplex $\conv V_J$ with
$\#J=d+1$ the set $\pos V^\star_{J^\star}\cap\Sph^{n-d-2}$ is a full-dimensional spherical
simplex, which is called the \emph{dual simplex} of $\conv V_J$.  The \emph{chamber
  complex} $\ChamberComplex{P}$ is the set of subsets of $\Sph^{n-d-2}$ arising from the
intersections of all the dual simplices.  The following theorem by Billera, Gel\cprime
fand, and Sturmfels \cite{BilleraSturmfels93} (see also \cite[\S5.3]{Triangulations}) is essential.

\begin{theorem}[{\cite[Theorem~3.1]{BilleraSturmfels93}}]\label{thm:secondaryDuality}
  The chamber complex $\ChamberComplex{P}$ is anti-isomorphic to the boundary complex of
  the secondary polytope $\SecondaryPolytope{P}$.
\end{theorem}

A \emph{split} of the polytope $P$ is a polytopal decomposition (without new vertices)
with exactly two maximal cells.  Splits are always regular.  The affine hyperplanes weakly
separating the two maximal cells of a split are characterized by the property that they do
not cut through any edges of $P$ \cite[Observation~3.1]{HerrmannJoswig08}; they are called
\emph{split hyperplanes}.  Two splits of $P$ are \emph{compatible} if their split
hyperplanes do not intersect in the interior of $P$.  They are \emph{weakly compatible} if
they admit a common refinement.  Clearly, compatibility implies weak compatibility, but
the converse is not true; see Example \ref{ex:splitcomplex:cross} below.  By definition
the splits are coarsest subdivisions of $P$ and hence correspond to rays in the secondary
fan or, equivalently, to facets of the secondary polytope and to vertices in the chamber
complex.  The \emph{split complex} $\SplitComplex{P}$ is the abstract flag-simplicial
complex whose vertices are the splits of $P$ which is induced by the compatibility
relation.  The \emph{weak split complex} $\WeakSplitComplex{P}$ is the subcomplex of
$\SecondaryFanOne{P}$ induced by the splits.

\begin{example}\label{ex:splitcomplex:cross}
  Let $P=\conv\smallSetOf{\pm e_i}{i\in[d]}$ be a regular crosspolytope in dimension~$d$.
  The splits of $P$ are given by the coordinate hyperplanes $x_i=0$, for $i\in[d]$. By
  combining any $d-1$ of these splits one gets a triangulation of $P$. This shows that the
  weak split complex is isomorphic to the boundary of a $(d-1)$-simplex.  However, any two
  coordinate hyperplanes contain the origin, whence the corresponding splits are not
  compatible.  The split complex of $P$ has $d$ isolated points.  See also
  \cite[Example~4.9]{HerrmannJoswig08}.
\end{example}

\begin{proposition}\label{prop:split-om}
  The split complex $\SplitComplex{P}$ and the weak split complex $\WeakSplitComplex{P}$
  of a polytope~$P$ only depend on the oriented matroid of $P$.
\end{proposition}

\begin{proof}
  Each split $S$ of $P$ defines a cocircuit $C^\star$ of the oriented matroid of $P$.  A
  hyperplane which separates $P$ defines a split if and only if it does not separate any
  edge of $P$.  However, an edge of $P$ is a covector of $P$ with exactly two positive
  entries and it is separated by $S$ if and only if one if the entries is in $C^\star_+$
  and the other is in $C_-^\star$. So one sees that the set of splits of $P$ only depends
  on the oriented matroid of $P$.
  
  Now it remains to show that also the compatibility and weak compatibility relations
  among splits only depend on the oriented matroid.
  
  Let $S_1$ and $S_2$ be two splits of $P$ with split hyperplanes $H_{S_1}$ and $H_{S_2}$,
  respectively.  Suppose that $S_1$ and $S_2$ are incompatible.  Then there exists a point
  $x\in\interior P\cap H_{S_1}\cap H_{S_2}$.  Since both split hyperplanes are spanned by
  vertices of $P$ and since, moreover, each split hyperplane does not intersect any edge
  the point $x$ is a convex combination of vertices of $P$ on $H_{S_1}$ as well as a
  convex combination of vertices of $P$ on $H_{S_2}$.  Thus $x$ gives rise to a vector $C$
  in the oriented matroid of $P$ such that $C_+$ is supported on vertices of $P$ lying on
  $H_{S_1}$ and $C_-$ is supported on vertices of $P$ lying on $H_{S_2}$.  That $x$ is
  contained in the interior of $P$ is equivalent to the property that $C_+\cup C_-$ is not
  contained in any facet of $P$.  Since the facets are precisely the positive cocircuits
  it follows that this can be read off from the oriented matroid of $P$.
   
  The statement for the weak split complex follows from the fact that one can construct
  common refinements of given subdivisions while only knowing the oriented matroid of the
  underlying polytope \cite[Corollary 4.1.43]{Triangulations}.
\end{proof}

Note that, of course, knowing the combinatorics, that is the face lattice of a polytope is
not enough for knowing its split complex or even its splits. As an example consider the
regular octahedron (with three splits; see Example~\ref{ex:splitcomplex:cross}) and an
octahedron with perturbed vertices (which does not have any split). Further, note that the
set of regular subdivisions of a polytope does not only depend on the oriented matroid but
rather on the coordinatization.  So the split subdivisions form a subset of all regular
subdivisions which is independent of the coordinatization.  In particular, the split
complex is a common approximation for the secondary fans of all polytopes with the same
oriented matroid but affinely inequivalent coordinates.  The next lemma explains how
splits can be recognized in the chamber complex.  We continue to use the notation
introduced above.  In particular, $P$ is the polytope and $G$ its spherical Gale dual.

\begin{lemma}\label{lem:unique-cycle}
  A point $x\in\Sph^{n-d-2}$ defines a split of $P$ if and only if there exists a unique
  circuit $C$ in $G$ such that $\pos x=\pos V^\star_{C_+} \cap \pos V^\star_{C_-}$.
\end{lemma}

\begin{proof}
  Consider $x\in\Sph^{n-d-2}$ such that its chamber is dual to a split $S$ of $P$.  Then
  the split hyperplane $H_S$ defines a unique cocircuit $C$ of $P$.  Equivalently, $C$ is
  a circuit of $G$.  Moreover, $\pos V^\star_{C_+}$ and $\pos V^\star_{C_-}$ correspond to
  the two maximal cells of $S$, and $\pos x=\pos V^\star_{C_+} \cap \pos
  V^\star_{C_-}$. Suppose that there is another circuit $C'$ in $G$ with the same
  property. Then the hyperplane $H$ defined by the elements of the corresponding cocircuit of
  $V_P$ separates the preimage of $x$ from all remaining vertices of $P$. However, since
  $x$ defines a split $S$ we get $H=H_S$ and hence the uniqueness.

  Conversely, let $C$ be the unique circuit of $G$ such that $\pos x=\pos V^\star_{C_+}
  \cap \pos V^\star_{C_-}$ for some $x\in\Sph^{n-d-2}$.  Obviously, $x$ is a ray of the
  chamber complex, and hence it is dual to a coarsest subdivision $S$ of $P$.  By
  \cite[Lemma~3.2]{BilleraSturmfels93}, the subdivision corresponding to $x$ has two
  maximal cells, since $\pos V^\star_{C_+}$ and $\pos V^\star_{C_-}$ are the only
  (necessarily minimal) dual cells containing~$x$.
  \end{proof}

\begin{example}\label{ex:pentagon:2}
  Let $P$ be the pentagon and $G$ its Gale dual from Example~\ref{ex:pentagon:1}.  Then
  $C=(0+0--)$ is a cocircuit of $P$ corresponding to the split defined by the line through
  the vertices $v_1$ and $v_3$.  Clearly, $C$ is also a circuit of $G$, with $C_+=\{2\}$
  and $C_-=\{4,5\}$.  We have $\pos v_2^\star=\pos V^\star_{\{2\}}\cap\pos
  V^\star_{\{4,5\}}$, and $C$ is the unique circuit of $G$ yielding $\pos v_2^\star$ as
  the intersection of its positive and its negative cone.  The two maximal cells of the
  split are the quadrangle $\conv V_{\{2\}^\star}$ and the triangle $\conv
  V_{\{4,5\}^\star}$.  See Figure~\ref{fig:pentagon}.
\end{example}

With each split $S$ of $P$ we associate the unique circuit $C[S]$ of $G$ from
Lemma~\ref{lem:unique-cycle}. If $V^\star_{C[S]_+}$ or ($V^\star_{C[S]_-}$) consists of a
single element $v^\star$ corresponding to a vertex $v$ of $P$, we call $S$ the
\emph{vertex split} for the vertex $v$ and also write $C[v]$ for $C[S]$. Note that the
support of $C[v]$ corresponds to the set of all vertices of $P$ that are connected to $v$
by an edge.

\begin{lemma}\label{lem:compatible-vertex-splits}
  Let $S$ and $S'$ be vertex splits with respect to vertices $v$ and $v'$ of $P$.  Then
  $S$ and $S'$ are compatible if and only if $v$ and $v'$ are not joined by an edge.
\end{lemma}

\begin{proof}
  It is easily seen that two splits $S,S'$ are compatible if and only if (possibly after
  the negation of one or both of the circuits) $C[S]_+\subseteq C[S']_+$ and
  $C[S']_-\subseteq C[S]_-$.  For a vertex split with respect to the vertex $v$ we have
  $C[v]_+=\{v^\star\}$ or $C[v]_-=\{v^\star\}$.  However, if $v$ and $v'$ are joined by an
  edge, then $v^\star \in C[v']_0$, so the above conditions cannot hold. On the other hand,
  if $v$ and $v'$ are not joined by an edge, and, say, $C[v]_+=\{v^\star\}$, then (possibly
  after a negation) $v^\star\in C[v']_+$ which implies $\{v^\star\}=C[v]_+\subseteq C[v']_+$.
\end{proof}

Clearly, $P$ admits a vertex split at the vertex $v$ if and only if the neighbors of $v$
in the vertex-edge graph of $P$ lie on a common hyperplane.  In particular, if $P$ is
simple then each vertex gives rise to a vertex split.

\section{Totally Splittable Polytopes}

We call a polytope \emph{totally splittable} if all regular triangulations of $P$ are
split triangulations.  We aim at the following complete characterization.

\begin{theorem}\label{thm:totally-splittable-classification}
  A polytope $P$ is totally splittable if and only if it has the same oriented matroid as
  a simplex, a crosspolytope, a polygon, a prism over a simplex, or a (possibly multiple)
  join of these polytopes.
\end{theorem}

By Proposition \ref{prop:split-om} the set of splits and their (weak) compatibility only
depends on the oriented matroid of $P$, and hence the notion ``totally splittable'' also
depends on the oriented matroid only.  The \emph{join} $P*Q$ of a $d$-polytope $P$ and an
$e$-polytope $Q$ is the convex hull of $P \cup Q$, seen as subpolytopes in mutually
skew affine subspaces of $\RR^{d+e+1}$.  For instance, a $3$-simplex is the join of any
pair of its disjoint edges.  In order to avoid cumbersome notation in the remainder of this
section we do not distinguish between any two polytopes sharing the same oriented matroid.
For instance, ``$P$ is a join of $P_1$ and $P_2$'' actually means ``$P$ has the same
oriented matroid as the join of $P_1$ and $P_2$'' and so on.

\begin{example}\label{ex:totally-splittable}
  We inspect the classes of polytopes occurring in
  Theorem~\ref{thm:totally-splittable-classification}.
  \begin{enumerate}
  \item Simplices are totally splittable in a trivial way.
  \item A triangulation of an $n$-gon is equivalent to choosing $n-3$ diagonals which are
    pairwise non-intersecting.  This is a compatible system of splits, and hence each
    polygon is totally splittable; see \cite[Example~4.8]{HerrmannJoswig08}. The secondary
    polytope of an $n$-gon is the $(n-3)$-dimensional associahedron \cite[Chapter~7,
    \S3.B]{GKZ94}.
  \item \label{ex:totally-splittable:cross} Let $P=\conv\smallSetOf{\pm e_i}{i\in[d]}$ be
    a regular crosspolytope in dimension~$d$ as in Example~\ref{ex:splitcomplex:cross}.
    The splits correspond to the coordinate hyperplanes, and any $d-1$ of them induce a
    triangulation of $P$.  Conversely, each triangulation of $P$ arises in this way.  See
    \cite[Example~4.9]{HerrmannJoswig08}.  A Gale dual of $P$ is given by the multiset $G
    \subset \Sph^{d-2}$ consisting of all points
    \[
    \bigSetOf{e_i}{i\in[d-1]}\cup\bigl\{-\frac{1}{\sqrt{d-1}}\sum_{i=1}^{d-1}e_i\bigr\} \, ,
    \]
    where each point occurs exactly twice.  All the vertices in the chamber complex
    correspond to vertex splits, and the chamber complex is the normal fan of a
    $(d-1)$-simplex (where each vertex carries two labels). So the secondary polytope of
    $P$ is a $(d-1)$-simplex. See Figure~\ref{fig:octahedron_prism} (left) below for
    $d=3$.
  \item Let $P$ be the prism over a $(d-1)$-simplex.  Then the dual graph of any
    triangulation of $P$ is a path with $d$ nodes.  The secondary polytope of $P$ is
    the $(d-1)$-dimensional permutohedron \cite[Chapter~7, \S3.C]{GKZ94}.  See
    Figure~\ref{fig:octahedron_prism} (right) below for $d=3$.
  \end{enumerate}
\end{example}

\begin{figure}[htb]
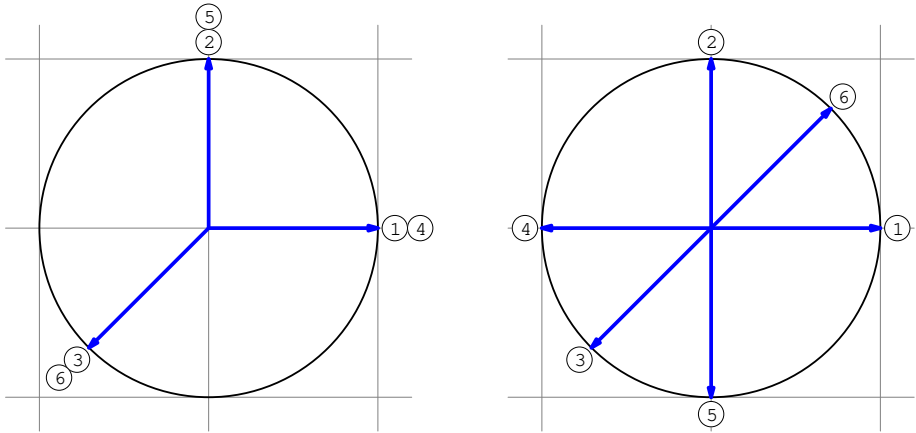

  \includegraphics[scale=.75]{gale.0} \qquad \includegraphics[scale=.75]{gale.1}
  \caption{Gale diagrams of the regular octahedron (left) and of the prism over a triangle
    (right).}
  \label{fig:octahedron_prism}
\end{figure}

\begin{remark}
  As the secondary polytope of a join of polytopes is the product of their secondary polytopes
  (e.g., this can be inferred from \cite[Corollary~4.2.8]{Triangulations}), 
  Theorem~\ref{thm:totally-splittable-classification} and Example~\ref{ex:totally-splittable} show
  that the secondary polytopes
  of totally splittable polytopes are (possibly multiple) products of simplices, permutohedra, and
  associahedra.
\end{remark}

\begin{remark}
  One can ask the question: What is the typical behavior of a polytope in terms of splits?
  The smallest example of a polytope that does not have any split is given by an
  octahedron whose vertices are slightly perturbed into general position.  Moreover, any
  \emph{$2$-neighborly polytope} (that is, any two vertices share an edge) does not admit
  any split \cite[Proposition~3.4]{HerrmannJoswig08}. On the other hand, $d$-dimensional
  simple polytopes with $n$ vertices have at least $n$ splits: Each vertex is connected to
  exactly $d$ other vertices which span a split hyperplane for the corresponding vertex
  split. This shows that the answer of the seemingly more precise question of how many
  splits is a ``random polytope'' expected to have highly depends on the chosen model. On
  the one hand, a $d$-polytope whose facets are chosen uniformly at random tangent to the
  unit sphere is simple with probability one; hence it has at least as many splits as
  vertices.  On the other hand one can choose models such that the polytopes generated are
  $2$-neighborly with high probability \cite{Shemer82}; such polytopes do not have any
  splits.
\end{remark}

It is obvious that total splittability is a severe restriction among polytopes.  The
following result is a key first step.  As an essential tool we use that any ordering of
the vertices of a polytope induces a triangulation, the \emph{placing triangulation} with
respect to that ordering \cite[\S4.3.1]{Triangulations}.  Moreover, successive placing of
new vertices can be used to extend any triangulation of a subpolytope.

\begin{proposition}\label{prop:vertex-splits}
  Let $P$ be a totally splittable polytope.  Then each face, each vertex figure, and each
  subpolytope $Q:=\conv (V\setminus\{v\})$ for a vertex $v\in V$ is totally splittable.
  Moreover, $v$ gives rise to a vertex split, and the neighbors of $v$ span a
  facet of $Q$.
\end{proposition}

\begin{proof}
  Let $\Delta$ be an arbitrary triangulation of a facet $F$ of $P$. We have to show that
  $\Delta$ is induced by splits of $F$.  By placing the vertices of $P$ not in $F$ in an
  arbitrary order we can extend $\Delta$ to a triangulation $\Delta'$ of $P$. As $P$ is
  totally splittable $\Delta'$ is induced by splits of $P$. A split of $P$ either does not
  separate $F$, or it is a split of $F$.  This implies that $\Delta$ is induced by splits
  of $F$.  Inductively, this shows the total splittability of all faces of $P$.
  
  Consider the subpolytope $Q:=\conv (V\setminus\{v\})$ for some vertex $v$ of $P$.  We
  can assume that $P$ is not a simplex, whence $Q$ is full-dimensional. Take an arbitrary
  triangulation $\Sigma$ of $Q$.  By placing $v$ this extends to a triangulation $\Sigma'$
  of $P$.  The $d$-simplices of $\Sigma'$ containing $v$ are the cones (with apex $v$)
  over those codimension $(d-1)$-faces of $\Sigma$ which span a hyperplane weakly
  separating $Q$ from $v$.  By assumption, $\Sigma'$ is a split triangulation, and hence
  each interior cell of codimension one spans a split hyperplane.  Fix a $d$-simplex
  $\sigma\in \Sigma'$ containing $v$.  The facet of $\sigma$ not containing $v$ is an
  interior cell of codimension one, which is why it spans a split hyperplane $H$.  Since
  $H$ cannot cut through the other simplices in $\Sigma'$ all neighbors of $v$ in the
  vertex-edge graph of $P$ are contained in $H$.  This proves that $H$ is the split
  hyperplane of the vertex split to $v$, and $H$ intersects $Q$ in a facet.  This also
  shows that the triangulation $\Sigma$ of $Q$ is induced by splits of $Q$, and $Q$ is
  totally splittable.

  The vertex figure of $P$ at $v$ is affinely equivalent to the facet $Q\cap H$ of $Q$,
  and hence the total splittability of the vertex figure follows from the above.
\end{proof}

\begin{remark}
  The same argument as in the proof above shows: Each hyperplane spanned by $d$ affinely
  independent vertices of a totally splittable polytope defines a facet or a split.
\end{remark}

Note that there exist polytopes for which each vertex defines a vertex split, but which
are not totally splittable.  An example is the $3$-cube which is simple, and hence each
vertex defines a vertex split \cite[Remark~3.3]{HerrmannJoswig08}, but which has several
triangulations which are not induced by splits \cite[Examples~3.8
and~4.10]{HerrmannJoswig08}. It is crucial that, by Proposition~\ref{prop:vertex-splits},
the neighbors of a vertex $v$ of a totally splittable polytope span a hyperplane, which we
denote by $v^\perp$.  two vertices of $P$ are \emph{neighbors} if they share an edge $w$
in the vertex-edge graph of $P$. Proposition~\ref{prop:vertex-splits} makes it possible to
re-read Lemma~\ref{lem:unique-cycle} as follows.

\begin{corollary}\label{cor:perp_intersection}
  Let $v$ be a vertex of a totally splittable polytope $P$.  Then
  \[
  v \ \in \ \bigcap_{\text{$w$ neighbor vertex to $v$}}w^\perp \, .
  \]
\end{corollary}

\begin{remark}\label{rem:beyond}
  In the situation of Proposition~\ref{prop:vertex-splits} all facets of $Q$ are also
  facets of $P$ except for the facet $F$ spanning the hyperplane $v^\perp$.  Moreover, all
  vertices of $Q$ are also vertices of $P$.  In this situation we say that $v$ is
  \emph{almost beyond} the facet $F$ of $Q$.  This is slightly more general than requiring
  $v$ to be \emph{beyond} $Q$, which means that $F$ is the unique facet of $Q$ violated by
  $v$, and additionally $v$ is not contained in any hyperplane spanned by a facet of $Q$.
  That $F$ is \emph{violated} by $v$ means that the closed affine halfspace with boundary
  hyperplane $\aff F$ does not contain the point $v$.  If $v$ is beyond $F$ and $d=\dim
  P=\dim Q\ge 3$ then the vertex-edge graph of $Q$ is the subgraph of the vertex-edge
  graph of $P$ induced on $\Vertices{P}\setminus\{v\}=\Vertices{Q}$.  The neighbors of $v$
  are precisely the vertices on the facet $F$ of $Q$.
\end{remark}

\begin{lemma}\label{lem:join}
  For two polytopes $P$ and $Q$ the join $P\join Q$ is totally splittable if and only if
  both $P$ and $Q$ are.
\end{lemma}

\begin{proof}
  Suppose that $P\join Q$ is totally splittable.  Then $P$ and $Q$ both occur as faces of
  $P\join Q$, and the claim follows from Proposition \ref{prop:vertex-splits}.

  Let $\dim P=d$ and $\dim Q=e$, and assume that $P$ and $Q$ both are totally splittable.
  The join of a $d$-simplex and an $e$-simplex is a $(d+e+1)$-simplex, and hence the join
  cell-by-cell of a triangulation of $P$ and a triangulation of $Q$ yields a triangulation
  of $P\join Q$.  Conversely, each triangulation of $P\join Q$ arises in this way
  \cite[Theorem~4.2.7]{Triangulations}.  The join of a split hyperplane of $P$ with $\aff Q$
  and the join of a split hyperplane of $Q$ with $\aff P$ yields split hyperplanes of $P\join
  Q$.  Now consider any triangulation $\Delta$ of $P\join Q$.  Then there are
  triangulations $\Delta_P$ and $\Delta_Q$ of $P$ and $Q$, respectively, such that
  $\Delta=\Delta_P\join\Delta_Q$.  By assumption, there is a set $S_P$ of splits of $P$
  inducing $\Delta_P$.  Likewise $S_Q$ is the set of splits inducing $\Delta_Q$.  Then the
  set of joins of all splits from $S_P$ with $\aff Q$ (as an affine subspace of
  $\RR^{d+e+1}$) and the set of joins of all splits from $S_Q$ with $\aff P$ jointly
  induce the triangulation~$\Delta$.
\end{proof}

Lemma \ref{lem:join} together with Example \ref{ex:totally-splittable} completes the proof
that all the polytopes listed in Theorem \ref{thm:totally-splittable-classification} are,
in fact, totally splittable.  The remainder of this section is devoted to proving that
there are no others.

\begin{proposition}\label{prop:cross-polytope}
  Let $P\subset\RR^d$ be a proper totally splittable $d$-polytope. Then $P$ is a regular
  crosspolytope if and only if the intersection $\bigcap_{v\in\Vertices P} v^\perp$ is not
  empty.
\end{proposition}

\begin{proof}
  Clearly, the regular crosspolytope $P=\conv\smallSetOf{\pm e_i}{i\in[d]}$ has the
  property that the intersection of its split hyperplanes is the origin.  Conversely,
  suppose that $P$ is not a crosspolytope. We assumed that $P$ is proper, meaning that $P$
  is not a pyramid.  Hence there exists a vertex $v$ of $P$ such that at least two
  vertices $u,w$ are separated from $v$ by the hyperplane $v^\perp$.  By
  Proposition~\ref{prop:vertex-splits}, the split hyperplane $v^\perp$ passes through the
  neighbors of $v$ in the vertex-edge graph of $P$.  Since $u$ is on the same side of
  $v^\perp$ as $w$ it follows that $v^\perp\ne w^\perp$ and, moreover, $v^\perp\cap
  w^\perp\cap\interior P=\emptyset$.  Now suppose that the intersection of all split
  hyperplanes contains points in the boundary of~$P$.  But since the split hyperplanes do
  not cut through edges, the intersection must contain at least one vertex $x\in\Vertices
  P$.  This is a contradiction since $x\not\in x^\perp$.  By a similar argument, we can
  exclude the final possibility that the intersection of all split hyperplanes contains
  any points outside $P$.  Therefore this intersection is empty, as we wanted to show.
\end{proof}

In a way crosspolytopes (which are not quadrangles) are maximally totally splittable.

\begin{lemma}\label{lem:cross_ext}
  Let $P\subset\RR^d$ be a $d$-dimensional regular crosspolytope and
  $v\in\RR^d\setminus P$ be a point almost beyond the facet $F$ of $P$.  If $d\ge 3$
  then $\conv(P\cup\{v\})$ is not totally splittable.
\end{lemma}

\begin{proof}
  Without loss of generality $P=\conv\{\pm e_1,\pm e_2,\dots,\pm e_d\}$.  Suppose that
  $\conv(P\cup\{v\})$ is totally splittable.  Since we assumed $d \ge 3$ each vertex $w$
  of $P$ has at least $d+1$ neighbors. At least $d$ affinely independent vertices among
  these are still neighbors of $w$ in $\conv(P\cup\{v\})$, so the hyperplane $w^\perp$
  with respect to $P$ is the same as $w^\perp$ with respect to $\conv(P\cup\{v\})$.  We
  have that $F^\perp:=\bigcap_{w\in\Vertices{F}}w^\perp=\{0\}$, which implies $v\not\in
  F^\perp$, a contradiction to Corollary~\ref{cor:perp_intersection}.
\end{proof}

\begin{figure}[hbt]\label{fig:prism-plus-one}
  \includegraphics[width=.37\textwidth]{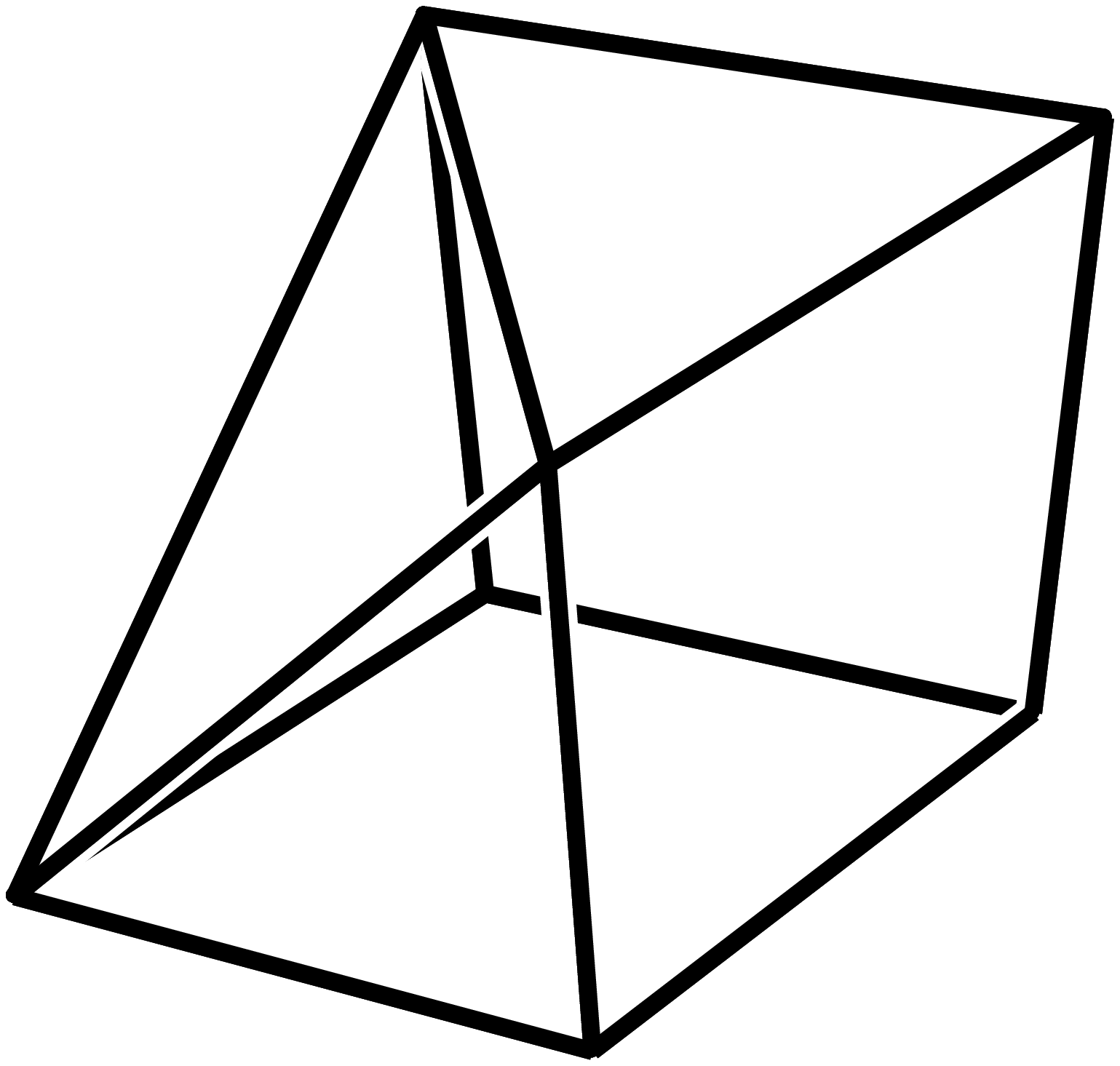}
  \quad
  \includegraphics[width=.37\textwidth]{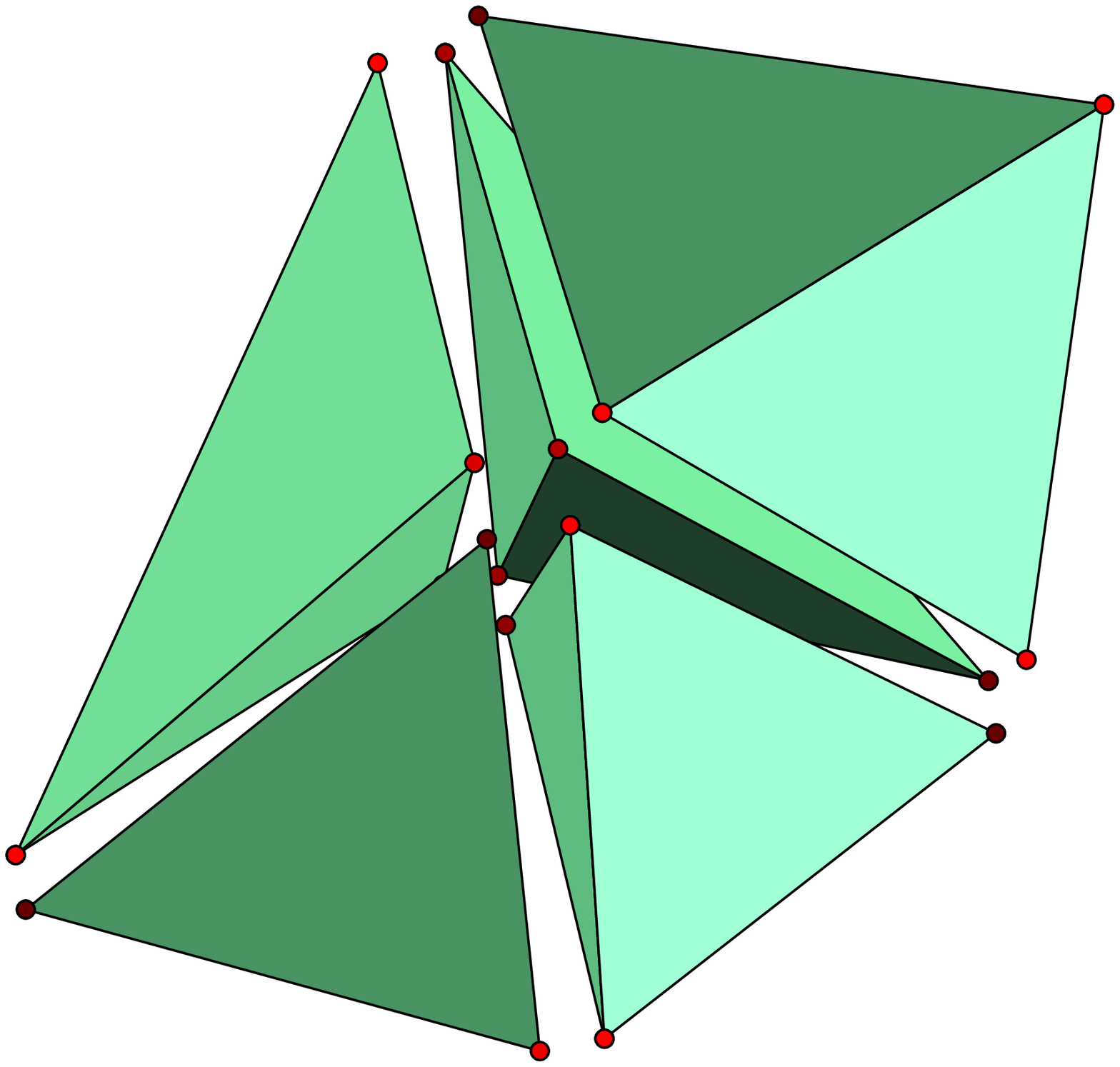}
  \caption{Convex hull of prism plus one point almost beyond a quadrangular facet,
    vertex-edge graph (left) and a non-split triangulation (right).}
\end{figure}

The same conclusion as in Lemma~\ref{lem:cross_ext} holds for prisms over simplices as
well.  See also Figure~\ref{fig:prism-plus-one} and Example~\ref{ex:prism-plus-one} below.

\begin{lemma}\label{lem:prism_ext}
  Let $P\subset\RR^d$ be a prism over a $(d-1)$-simplex and $v\in\RR^d\setminus P$ a point
  which is almost beyond a facet $F$ of $P$.  If $d\ge 3$ then $\conv(P\cup\{v\})$ is not
  totally splittable.
\end{lemma}

\begin{proof}
  Suppose that $\conv(P\cup\{v\})$ is totally splittable.  As in the proof of
  Lemma~\ref{lem:cross_ext} we are aiming at a contradiction to
  Corollary~\ref{cor:perp_intersection}.  First suppose that $v$ is beyond $F$, and hence
  for $w\in\Vertices{P}$ the hyperplanes $w^\perp$ with respect to $P$ and
  $\conv(P\cup\{v\})$ coincide, since $d\ge 3$; see Remark~\ref{rem:beyond}.  Up to an
  affine transformation we can assume that
  $P=\conv\{e_1,e_2,\dots,e_d,f_1,f_2,\dots,f_d\}$ with
  \[
  f_k \ = \ -\sum_{i\ne k} e_i \, .
  \]
  The neighbors of the vertex $e_k$ are $e_1,e_2,\dots,e_{k-1},e_{k+1},\dots,e_d$ and
  $f_k$; symmetrically for the $f_k$. A direct computation shows that
  \[
  e_k^\perp\ =\ \SetOf{x}{x_k=0} \quad \text{and} \quad f_k^\perp \ = \SetOf{x}{2\sum_{i\ne k} x_i=(d-2)(x_k-1)}\,.
  \]
  
  We have to distinguish two cases: the facet $F$ of $P$ violated by $v$ may be a
  $(d-1)$-simplex or a prism over a $(d-2)$-simplex.  If $F$ is a simplex, for instance,
  $\conv\{e_1,e_2,\dots,e_d\}$, then we can conclude that the set $\bigcap_{w\in F}
  w^\perp = \{0\}$ which is in the interior of $P$ and hence cannot be equal to $v$.  If,
  however, $F$ is a prism, for instance, with the vertices
  $e_1,e_2,\dots,e_{d-1},f_1,f_2,\dots,f_{d-1}$, we can compute that
  \[
  \bigcap_{w\in\Vertices{F}} w^\perp \ = \ \left\{\frac{2-d}2 e_d \right\} \,,
  \]
  again an interior point.
  In both cases we arrive at the desired contradiction to
  Corollary~\ref{cor:perp_intersection}.
  
  Now suppose that $v$ violates $F$ but it is not beyond $F$, that is, $v$ is contained in
  the affine hull of some facet $F'$ of $P$.  Let us assume that $d\geq 4$ and that the
  assertion is true for $d=3$. Then the polytope $\conv(F'\cup\{v\})$ is totally
  splittable by Proposition~\ref{prop:vertex-splits}. Again, $F'$ may be a $(d-1)$-simplex
  or a prism over a $(d-2)$-simplex. If $F'$ is a $(d-1)$-simplex, it can easily be seen
  that $\conv(F'\cup\{v\})$ is not totally splittable for $d>3$ since $F'$ does not have
  any splits.  If $F'$ is a prism over a simplex, we are done by induction.
  
  An easy consideration of the cases, which we omit, allows us to prove the result in the base
  case $d=3$. See Example~\ref{ex:prism-plus-one} and Figure~\ref{fig:prism-plus-one} for
  one of the cases arising.
\end{proof}

\begin{example}\label{ex:prism-plus-one}
  Consider the $3$-polytope $P=\conv\{e_1,e_2,e_3,-e_2-e_3,-e_1-e_3,-e_1-e_2\}$, which is
  a prism over a triangle.  For instance, the point $v=e_1+e_2-e_3$ lies almost beyond the
  quadrangular facet $F=\conv\{e_1,e_2,-e_2-e_3,-e_1-e_3\}$.  The polytope
  $\conv(P\cup\{v\})$ admits a triangulation which is not split; see
  Figure~\ref{fig:prism-plus-one}.
\end{example}

\begin{proposition}\label{prop:join-cycle-disjoint}
  Let $P$ be a proper totally splittable polytope that is not a regular crosspolytope.
  Then $P$ is a join if and only if the vertex set of $P$ admits a partition $\Vertices
  P=U\cup W$ such that no vertex split of a vertex in $U$ is compatible with any vertex
  split of a vertex in~$W$.
\end{proposition}

\begin{proof}
  Let $P=(\conv U) * (\conv W)$ be a proper join.  In particular, $P$ is not a pyramid,
  and $\conv U$ and $\conv W$ both are at least one-dimensional.  Then, by the definition
  of join, each vertex in $U$ shares an edge with each vertex in $W$, and thus the
  corresponding vertex splits are not compatible.

  Conversely, assume that no split with respect to a vertex in $U$ is compatible with a
  split with respect to any vertex in $W$.  By Lemma~\ref{lem:compatible-vertex-splits}
  each vertex in $U$ is joined by an edge to each vertex in $W$.
  Proposition~\ref{prop:vertex-splits} says that each vertex split hyperplane $u^\perp$
  contains all neighbors of $u$.  Thus we infer that $\bigcap_{u\in U} u^\perp\supset
  \conv W$ and, symmetrically, $\bigcap_{w\in W} w^\perp\supset \conv U$.  Now there are
  two cases to distinguish.  If $\bigcap_{v\in\Vertices P} v^\perp$ is non-empty then $P$
  is a regular crosspolytope due to Proposition~\ref{prop:cross-polytope} contradicting
  the assumption.  The remaining possibility is that $\bigcap_{v\in\Vertices P} v^\perp$
  is empty.  In this case we have
  \[
  \aff U\cap \aff W \ \subseteq \ \bigcap_{w\in W} w^\perp \cap \bigcap_{u\in U} u^\perp \
  = \ \bigcap_{v\in\Vertices P} v^\perp \ = \ \emptyset \, .
  \]
  The affine subspaces $\aff U$ and $\aff W$ are skew.  It follows that $P=(\conv
  U)*(\conv W)$.
\end{proof}

For the following we will switch from the primal view on our polytope $P$ to its spherical Gale dual
$G$.  A point of multiplicity two in $G$ is called a \emph{double point}.  Vertices of $P$
corresponding to the same point in $G$ are called \emph{siblings}.

\begin{lemma}\label{lem:join-double-points}
  Let $P$ be a totally splittable polytope which is not a join, and let $G$ be a spherical
  Gale diagram of $P$.  Then $P$ is proper, and each point of $G$ is a single point, or
  each point is a double point.  In particular, there are no points in $G$ with
  multiplicity greater than two.
\end{lemma}

\begin{proof}
  If $P$ is a regular crosspolytope we know from the explicit description of $G$ in
  Example~\ref{ex:totally-splittable}~\eqref{ex:totally-splittable:cross} that the
  conclusion of the lemma holds. So we can assume that this is not the case. Since we
  assume that $P$ is not a join, in particular, it is not a pyramid, and this is why $P$
  is proper.  If $G$ had a point with multiplicity three or above, then each pair of
  copies of $x$ defines a circuit which yields a contradiction to
  Lemma~\ref{lem:unique-cycle}.

  So suppose now that $v_1$ is a vertex that has a sibling $v_2$ and that the set $W$ of
  all vertices without a sibling is non-empty.  Then, again by
  Lemma~\ref{lem:unique-cycle}, $v_1^\star=v_2^\star$ is not contained in $\pos W^\star$.
  By the Separation Theorem \cite[2.2.2]{Gruenbaum03}, there is an affine hyperplane in
  $\RR^{n-d-1}$ which weakly separates $v_1^\star=v_2^\star$ from $\pos W^\star$.  This
  argument even works for all vertices with a sibling simultaneously.  That is $H$ weakly
  separates the double points from non-double points.  By rotating $H$ slightly, if
  necessary, we can further assume that $H$ contains at least one dual vertex $w^\star$ of
  a vertex $w\in W$ without a sibling.  For each such $w\in W$ with $w^\star\in H$ the
  support of the circuit $C[w]$ is a subset of $W^\star$ and from
  Lemma~\ref{lem:unique-cycle} it follows that the support of $C[w]$ is contained in the
  hyperplane $H$. In the primal view, this means that all vertices $v$ of $P$ with
  $v^\star\not\in H$ have to be in the splitting hyperplane $w^\perp$ and that the vertex
  split of $w$ cannot be compatible to any vertex split of a vertex $v$ with
  $v^\star\not\in H$. If now we define $U:=\smallSetOf{w\in\Vertices P}{w^\star\in H}$ we
  have a partition of $\Vertices P$ in $U$ and $\Vertices P\setminus U$ such that no
  vertex split of a vertex in $U$ is compatible with any vertex split of a vertex in
  $\Vertices P\setminus U$. So $P$ is a join by
  Proposition~\ref{prop:join-cycle-disjoint}.
\end{proof}

\begin{figure}[htb]
  \includegraphics[scale=.75]{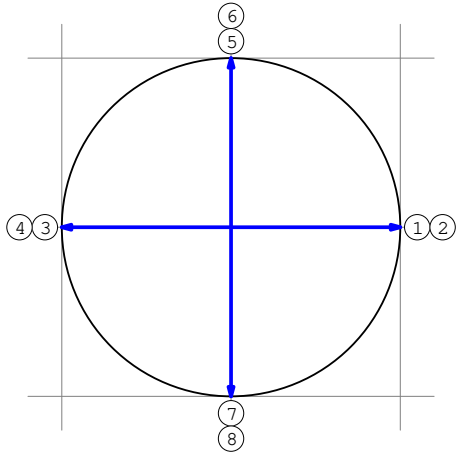}

  \caption{Gale diagram of the join of two squares, labeled $\{1,2,3,4\}$ and
    $\{5,6,7,8\}$, respectively.}
\end{figure}

A point $x\in G$ is \emph{antipodal} if $-x$ is also in $G$.  Notice that any quadrangle,
regular or not, has a zero-dimensional spherical Gale diagram with exactly two pairs of
antipodal points.

\begin{lemma}\label{lem:Gale-cross}
  Let $P$ be a totally splittable $d$-polytope with $d\ge 2$ which is not a join.  If each
  point in the spherical Gale diagram $G$ is a double point then $P$ is a regular
  crosspolytope.
\end{lemma}
  
\begin{proof}
  Assume that each point in $G$ is a double point.  Let $v$ be any vertex of $P$ and
  $v^\perp$ the hyperplane corresponding to the vertex split of $v$.  Since $v^\star$ is a
  double point in $G$ there is exactly one vertex $w$ other than $v$ which is not
  contained in $v^\perp$.  The polytope $Q:=\conv(\Vertices P\setminus\{v,w\})=P\cap
  v^\perp$ is a face of the vertex figure of $v$ and hence totally splittable by
  Proposition~\ref{prop:vertex-splits}.  Clearly, a spherical Gale diagram of $Q$ again
  has only double points.  Inductively, we can thus assume that $Q$ is a regular
  crosspolytope.  Therefore, its split hyperplanes have a non-empty intersection.  Since
  this intersection is contained in $v^\perp$ it follows that the split hyperplanes of $P$
  also have a non-empty intersection. Hence $P$ is a regular crosspolytope by
  Proposition~\ref{prop:cross-polytope}.  As a basis of the induction we can consider the
  case where $G$ is contained in $\Sph^1$.  As $G$ must span $\RR^2$, and as each point in
  $G$ occurs twice, the polytope $P$ has six vertices, and it is three-dimensional.  It
  can be shown that $P$ is a regular octahedron.  The two-dimensional case will be dealt
  with in the proof of Lemma~\ref{lem:Gale-both} below.
\end{proof}

\begin{lemma}\label{lem:Gale-prism-simplex}
  Let $P$ be a totally splittable $d$-polytope with $d\ge 2$ which is not a join.  If each
  point in the spherical Gale diagram $G$ is antipodal then $P$ is a  prism over a simplex.
\end{lemma}

\begin{proof}
  Suppose that each point in $G$ is antipodal.  Let $k:=n-d-1$ be the dimension of the
  linear span of $G$, and so we can view $G$ as a subset of $\Sph^{k-1}$.  We claim that
  the number of vertices of $P$ equals $n=2d$ or, equivalently, that $n=2k+2$.  Pick any
  cocircuit of $G$.  This corresponds to a linear hyperplane $H$ in $\RR^k$ which contains
  at least $2k-2$ points of $G$, due to antipodality.  Since $G$ is the Gale diagram of a
  polytope each open halfspace defined by $H$ contains at least $2$ points
  \cite[Theorem~6.19]{Ziegler95}.  We conclude that $n\ge 2k+2$.

  Now we will show that $n\leq2k+2$ hence $n=2k+2$.  To arrive at a contradiction, suppose
  that the spherical Gale diagram $G$ contains at least $k+2$ antipodal pairs.  Take any
  vertex $v$ of $P$, and let $v^\star$ be its dual in $G$.  Pick an affine hyperplane
  $H^\star$ in $\RR^k$ which is orthogonal to $v^\star$ and such that $v^\star$ and the origin
  are on different sides of $H^\star$. Let $W=\{v_1^\star,v_2^\star,\dots,v_m^\star\}$ be the set of
  points in $G$ distinct from $v^\star$ for which the corresponding rays
  intersect~$H^\star$.  Firstly, $m\ge k+1$ since $G$ contains $k+1$ antipodal pairs in
  addition to $v^*$ and its antipode.  Secondly, $v^\star$ is in the positive span of the
  rays corresponding to the points in $W$ since among those points are the elements of $C[v]_+$.
  By Carath\'eodory's Theorem \cite[\S2.3.5]{Gruenbaum03} we can assume
  that the corresponding rays of $v_1^\star,v_2^\star,\dots,v_{k+1}^\star$ still contain
  $v^*$ in their positive span.  Let $Q$ be the convex hull of the intersections of the
  rays corresponding to $v_1^\star,v_2^\star,\dots,v_{k+1}^\star$ with the hyperplane
  $H^\star$.  Now $Q$ is a $(k-1)$-dimensional polytope with $k+1$ vertices.  Such a
  polytope has precisely two triangulations $\Delta$ and $\Delta'$; these are related by a
  \emph{flip}, see \cite[\S2.4.1]{Triangulations}. Let $\sigma$ and $\sigma'$ be maximal
  simplices of $\Delta$ and $\Delta'$ containing the point $(\RR v^\star)\cap H^\star$.
  By construction $\sigma$ gives rise to a circuit $D$ of $G$ whose negative support
  corresponds to the vertices of $\sigma$ and its positive support corresponds to
  $v^\star$.  Similarly, $\sigma'$ defines another such circuit $D'$.  Since no maximal
  simplex of $\Delta$ also occurs as a maximal simplex in $\Delta'$ we have
  $\sigma\ne\sigma'$ implying $D\ne D'$.  This contradicts Lemma \ref{lem:unique-cycle},
  and this finally proves that $n$ equals $2k+2$.

  By now we know that $G$ consists of precisely $k+1$ antipodal pairs in $\Sph^{k-1}$.  So
  $P$ is a $d$-polytope with $2k+2=2d$ vertices. We have to show that $P$ has the same
  oriented matroid as a prism over a $(d-1)$-simplex. This will be done by showing that
  the cocircuits of $G$ (which are the circuits of $P$) agree with the circuits of a prism
  over a simplex. So consider the prism over a simplex with coordinates as in the proof of
  Lemma~\ref{lem:prism_ext}.  Then each circuit $C$ of this prism is of the form
  \[
  C_+=\{e_i,f_j\} \quad \text{and} \quad C_-=\{f_i,e_j\}
  \]
  for distinct $i$ and $j$.  Moreover, $e_i^\star$ and $f_i^\star$ are antipodes in the
  prism's spherical Gale diagram.  The cocircuits of $G$ are given by all (linear)
  hyperplanes in $\RR^k$ spanned by $k-1$ pairs of points in $G$. None of the other two
  pairs of points can be contained in such a hyperplane since $G$ is the Gale diagram of a
  polytope \cite[Theorem~6.19]{Ziegler95}. So the cocircuits of $G$ are given by
  $C^\star_+=\{x,y\}$, $C^\star_-=\{-x,-y\}$ for all distinct $x,y\in G$ with $x\not=\pm
  y$.
%

\end{proof}

\begin{lemma}\label{lem:Gale-both}
  Let $P$ be a totally splittable $d$-polytope with $d\ge 2$ which is not a join. If each
  point in the spherical Gale diagram $G$ is both a double point and antipodal then $d=2$,
  and $P$ is a quadrangle.
\end{lemma}

\begin{proof}
  If each point in $G$ is antipodal from Lemma~\ref{lem:Gale-prism-simplex} we know that
  $P$ is a prism over a $(d-1)$-simplex.  The only case in which such a Gale diagram has
  the property that each point is a double point is $d=2$, and $P$ is a quadrangle.
\end{proof}

Now we have all ingredients to prove our main result.

\begin{proof}[Proof of Theorem \ref{thm:totally-splittable-classification}]
  Let $P$ be a totally splittable $d$-polytope with spherical Gale dual $G$.  By
  Lemma~\ref{lem:join}, we can assume without loss of generality that $P$ is not a join.
  Consider a vertex $v\in \Vertices P$ with the property that $v^\star$ is neither a
  double nor an antipodal point. By Proposition~\ref{prop:vertex-splits}, the polytope
  $Q:=\conv(\Vertices P\setminus\{v\})$ obtained from $P$ by the deletion of $v$ is again
  totally splittable.  Moreover, $\dim Q=d$ since $P$ is not a pyramid.

  Let us assume for the moment that $Q$ is also not a join.  Then we can repeat this
  procedure until after finitely many steps we arrive at a polytope $P'$ with a spherical
  Gale diagram $G'$ which consists only of double and antipodal points.  In this situation
  Lemma~\ref{lem:join-double-points} implies that all points of $G'$ are double points or
  all points of $G'$ are antipodal.  Combining Lemma~\ref{lem:Gale-cross},
  Lemma~\ref{lem:Gale-prism-simplex}, and Lemma~\ref{lem:Gale-both}, we can conclude that
  either $d=\dim P=\dim P'=2$ and $P'$ is a quadrangle, or $d\ge 3$ and $P'$ is a regular
  crosspolytope, or $d\ge 3$ and $P'$ is a prism over a simplex.  The question remaining
  is whether $P$ and $P'$ can actually be different.  For $d\ge 3$ this is ruled out by
  Lemma~\ref{lem:cross_ext} (if $P'$ is a crosspolytope) and Lemma~\ref{lem:prism_ext} (if
  $P'$ is a prism).  In the final case $\dim P=\dim Q=\dim P'=2$.

  The proof of our main result will be concluded with the subsequent proposition.
\end{proof}

\begin{proposition}\label{prop:final}
  Let $P$ be a totally splittable polytope with spherical Gale diagram $G$, and let $v$ be
  a vertex of $P$ with the property that its dual $v^\star$ in $G$ is neither a double nor
  an antipodal point. If $P$ is not a join then neither is $Q:=\conv(\Vertices
  P\setminus\{v\})$.
\end{proposition}

\begin{proof}
  By \cite[Lemma 3.4]{BilleraSturmfels93}, the Gale transform of $Q$ is the minor
  $G/v^\star$ obtained by contracting $v^\star$ in $G$.  Up to an affine transformation we
  can assume that $v^\star$ is the first unit vector in $\RR^{n-d-1}$, and so $G/v^\star$
  is the projection of $G\setminus\{v^\star\}$ to the last $n-d-2$ coordinates.  We call
  the projection map $\pi$.  Since $v^\star$ is neither antipodal nor a double point, no
  point in $G/v^\star$ is a loop, and thus $Q$ is proper, that is, it is not a pyramid.
  
  So suppose that $Q=Q_1\join Q_2$ is a join with $\dim Q_1\ge1$ and $\dim Q_2\ge 1$.
  Then there are spherical Gale diagrams $G_1$ and $G_2$ of $Q_1$ and $Q_2$, respectively,
  such that $G/v^\star=G_1\sqcup G_2$ as a multiset in $\Sph^{n-d-3}$.  Up to exchanging
  the roles of $Q_1$ and $Q_2$, there is a facet $F_1$ of $Q_1$ such that $v^\perp\cap P$,
  which is a facet of $Q$, is a join $F_1*Q_2$.  That is to say, the cosupport of the
  circuit $C[v]$, corresponding to the vertex split of $v$ in $P$, is mapped to $G_1$ by
  $\pi$. In particular, $v^\star$ is not in the positive hull of the points dual to the
  vertices of $Q_2$.  The Separation Theorem \cite[2.2.2]{Gruenbaum03} implies that there
  is a linear hyperplane $H$ in $\RR^{n-d-1}$ separating $v^\star$ from the duals of the
  vertices of $Q_2$. As in the proof of Lemma \ref{lem:join-double-points} we can now
  argue that $P$ is a join, which contradicts our assumptions.
\end{proof}

This finally completes the proof of the theorem.

\begin{remark}
  If $v^\star$ is antipodal or a double point, then $Q$ is a pyramid over the unique facet
  of $Q$ which is not a facet of $P$.  This shows that the assumption on $v^\star$ in
  Proposition~\ref{prop:final} is necessary.  For instance, by inspecting the two Gale
  diagrams in Figure~\ref{fig:octahedron_prism} one can see directly that if $P$ is a
  regular octahedron or a prism over a triangle, in both cases $Q$ is a pyramid over a
  quadrangle.
\end{remark}

\begin{remark}
  A triangulation $\Delta$ of a $d$-polytope is \emph{foldable} if the dual graph of
  $\Delta$ is bipartite.  This is equivalent to the property that the $1$-skeleton of
  $\Delta$ is $(d+1)$-colorable.  In \cite[Corollary 4.12]{HerrmannJoswig08} it was proved
  that any triangulation generated by splits is foldable.  This means that each
  triangulation of a totally split polytope is necessarily foldable.
\end{remark}

We are indebted to Raman Sanyal for sharing the following observation with us.

\begin{corollary}\label{cor:equidecomposable}
  Each totally splittable polytope is equidecomposable.
\end{corollary}

A polytope is \emph{equidecomposable} if each triangulation has the same $f$-vector.

\begin{proof}
  This follows from the classification case by case: Each triangulation of an $n$-gon has
  exactly $n-2$ triangles. Each triangulation of a $d$-dimensional regular crosspolytope
  has exactly $2d-2$ maximal cells.  Each triangulation of a prism over a $(d-1)$-simplex
  has exactly $d$ maximal cells.  A similar count can be done for the lower dimensional
  cells.  Observe that equidecomposability is preserved under taking joins.
\end{proof}

It would be interesting to know if Corollary~\ref{cor:equidecomposable} has a direct proof
without relying on Theorem~\ref{thm:totally-splittable-classification}.

\begin{remark}
  Bayer \cite{Bayer93} defines a polytope to be \emph{weakly neighborly} if any $k$ of its
  vertices are contained in some face of dimension $2k-1$.  She shows that a weakly
  neighborly polytope is necessarily equidecomposable \cite[Corollary~10]{Bayer93}.
  Prisms over simplices are weakly neighborly whereas crosspolytopes are not; so the
  approach of Bayer is somewhat transverse to ours.  Moreover, all circuits of a totally
  splittable polytope are \emph{balanced} in the sense that the positive and the negative
  support share the same cardinality. This relates to the question of whether a polytope
  all of whose circuits are balanced is always equidecomposable.  The converse is true
  \cite[Theorem~1]{Bayer93}.
\end{remark}

\bibliographystyle{amsplain}
\bibliography{main}

\end{document}